\documentclass[12pt,a4paper]{article}

\def\R{\mathbb{R}}
\def\E{\mathbb{E}}
\def\P{\mathbb{P}}

\def\D{{\mathcal D}}

\newenvironment{theorem}[1]{\endgraf  {\bf #1}\em}{ \endgraf  }
\newenvironment{proof}[1][] {\noindent {\bf Proof#1:} }{\hspace*{\fill}$\square$\medskip\par}
\usepackage{amssymb, amsmath}

\begin{document}

\title{An Inequality Related to Bifractional Brownian Motion}

\author{Mikhail Lifshits \and Ilya Tyurin}

\date{}

\maketitle

\noindent {\it St.Petersburg State University, Dept Math.Mech., 198504 Stary Peterhof,
Bibliotechnaya pl. 2, Russia} 

\noindent Email: {\tt lifts@mail.rcom.ru}
\smallskip

\noindent {\it Moscow State University, Dept Mech.Math., 119991 Moscow,
Leninskie gory 1, Russia} 

\noindent Email: {\tt itiurin@gmail.com}
\bigskip

\medskip

\bigskip

\bigskip

\begin{abstract}
We prove that for any pair of i.i.d. random variables $X, Y$ with finite moment of order
$\alpha \in (0,2]$ it is true that
$$
 \E |X-Y|^\alpha \leq \E |X+Y|^\alpha.
$$
Surprisingly, this inequality turns out to be related with bifractional Brownian motion.
We extend this result to Bernstein functions and provide some counter-examples. 

\end{abstract}

\vskip 1 cm

\section{Introduction}
Let $X,Y$ be i.i.d. random variables with finite expectations. Then it is true that
\begin{equation} \label{e1}
  \E |X-Y| \leq \E |X+Y|.
\end{equation}
Inequality (\ref{e1}) appeared recently in the analytic context (properties of integrable functions)
\cite{XXX}. Since
(\ref{e1}) is a nice fact itself and it seems not to be well known in the probabilistic world, 
it is desirable to search
for adequate proofs and consider eventual extensions for it. In particular, is it true that
\begin{equation} \label{e2}
  \E |X-Y|^\alpha \leq \E |X+Y|^\alpha,
\end{equation}
provided $X$ and $Y$ are i.i.d. and $\E |X|^\alpha<\infty$?

Proving  (\ref{e1}) is a non-trivial exercise for a Probability course. In particular,
(\ref{e1}) follows from the identity
\[
   \E |X+Y| - \E |X-Y| =2 \int_0^\infty \left[\P(X>r)-\P(X<-r) \right]^2 dr.
\]
However, we are not aware of a similar elementary approach to (\ref{e2}).
In this note, we show how (\ref{e2}) suddenly emerges from some recent
advances in the theory of random processes.

\section{Main result}

We prove the following result.
\begin{theorem}{Theorem 1.} Let $\alpha\in (0,2]$.
For any pair of i.i.d. random variables $X, Y$ such that $\E|X|^\alpha<\infty$
it is true that
$$
  \E |X-Y|^\alpha \leq \E |X+Y|^\alpha.
$$
\end{theorem}

\begin{proof}
For the proof we will need the definition of the bifractional
Brownian motion introduced by Houdr\'e and Villa in
\cite{HV2003}. It is a centered Gaussian process $B^{H,K} =
\left(B^{H,K}_t, t\geq 0\right)$ with covariance function
$$
R^{H,K}(t,s) = \E\left(B^{H,K}_t B^{H,K}_s\right) =
\frac{1}{2^K}\left((t^{2H}+s^{2H})^K - |t-s|^{2HK}\right), \enskip
t,s\ge 0.
$$
Note that letting $K=1$ yields a usual fractional Brownian motion
$B^H$. Originally, the process was defined for the parameters
$H\in(0,1]$ and ${K\in(0,1]}$.
Bardina and Es-Sebaiy \cite{BE2010} proved recently that  $B^{H,K}$ exists
provided that $(H,K)\in \D$, where
\[
\D:= \{H,K: 0<H\le 1, 0<K\le 2, HK \le 1\}.
\]
(Actually a hint for this extension is contained in a preceding work
by Lei and Nualart \cite{LN2009} who established an integral representation
relating $B^{H,K}$ with fractional Brownian motion $B^{HK}$).

Now we relate our problem to a particular bifractional Brownian
motion. Note that for any $u, v \in \mathbb R$ and
$\alpha\in(0,2]$ one has
$$|u+v|^\alpha - |u-v|^\alpha = \begin{cases}(|u|+|v|)^\alpha -
||u|-|v||^\alpha, & uv\ge 0
\\ ||u|-|v||^\alpha-(|u|+|v|)^\alpha, &
uv < 0\end{cases},$$ which can be briefly expressed as
$$
\left((|u|+|v|)^\alpha - ||u|-|v||^\alpha\right)sgn(u)sgn(v) = 2^\alpha
R^{{1}/{2},\alpha}(|u|,|v|)sgn(u)sgn(v).
$$
Let $\left(B^{1/2,\alpha}_t, t\ge 0\right)$ be a bifractional
Brownian motion with ${H=1/2}$ and $K=\alpha$. Set
$$
Z_\alpha :=
\int^{+\infty}_{-\infty}B^{1/2,\alpha}_{|u|}sgn(u)P(du),
$$
where $P$ stands for the common distribution of $X$ and $Y$. Then
\begin{equation*}
\begin{split}
0 &\le Var(Z_\alpha) =
\int^{+\infty}_{-\infty}\int^{+\infty}_{-\infty} R^{1/2,\alpha}(|u|,|v|) sgn(u) sgn(v) P(du)P(dv)
\\ &= \int^{+\infty}_{-\infty}\int^{+\infty}_{-\infty} \frac{1}{2^\alpha}(|u+v|^\alpha -
|u-v|^\alpha)P(du)P(dv)
\\ &= \frac{1}{2^\alpha} [\E |X+Y|^\alpha - \E |X-Y|^\alpha].
\end{split}
\end{equation*}
\end{proof}

\section{Further extensions and counterexamples}

Inequality (\ref{e2}) trivially extends to the case $\alpha=\infty$ in the following sense.
Let
\begin{eqnarray*}
  M &=& \sup\{r: \P(X<r)<1 \};  \\
  m &=&  \sup\{r: \P(X<r)=0 \}.
\end{eqnarray*}
Then
\[
   ||X-Y||_\infty = M-m \le 2 \max\{|M|; |m| \} = ||X+Y||_\infty.
\]
\medskip

However, without further assumptions (\ref{e2}) does not hold in general for $2<\alpha<\infty$.
To see this,  let us fix $\alpha\in(2,\infty), c>0$. For any $M\ge c$ let $q=\frac cM, p=1-q$.
Let $X_M,Y_M$ be i.i.d. random variables such that
\begin{eqnarray*}
  \P(X_M=1) &=& \P(Y_M=1)=p; \\
  \P(X_M=-M) &=&\P(Y_M=-M)=q.
\end{eqnarray*}
If $M\ge 1$, then
\begin{eqnarray*}
&& \E |X_M-Y_M|^\alpha - \E |X_M+Y_M|^\alpha   
\\
&=& 2pq\left[ (M+1)^\alpha-(M-1)^\alpha \right]
-2^\alpha M^\alpha q^2 -2^\alpha p^2
\\
&\ge&  4pq \alpha M^{\alpha-1} -2^\alpha M^\alpha q^2-2^\alpha p^2
\\
   &=&  M^{\alpha-2}(4p\alpha c -2^\alpha c^2) -2^\alpha p^2.
\end{eqnarray*}
Hence, whenever $c<2^{2-\alpha} \alpha$ and $M$ is large enough,
\[
  \E |X_M-Y_M|^\alpha - \E |X_M+Y_M|^\alpha  >0
\]
and (\ref{e2}) fails.
\bigskip

It is natural to ask for which class of functions $F(\cdot)$ inequality  (\ref{e2}) extends to
\begin{equation} \label{e3}
  \E \, F(|X-Y|) \leq \E \,  F(|X+Y|).
\end{equation}
Recall that a function $G:\R_+\to\R_+$ is called {\it Bernstein function}, if it admits a representation
\[
  G(\lambda)= a+b\lambda+\int_0^\infty \left(1-e^{-t\lambda} \right) \mu(dt),
\]
where  $a,b\ge 0$, and a measure $\mu$ on $(0,\infty)$ satisfies integrability condition
$\int_0^\infty \min(t,1)\mu(dt)<\infty$. 
In a recent monograph \cite{SSV} one can find an exhaustive information about the importance of Bernstein
functions in Analysis and Probability, their properties, as well as a large list of such functions.
\bigskip

\begin{theorem}{Theorem 2.} Let $G$ be a Bernstein function. Then $F(\lambda):=G(\lambda^2)$ satisfies 
$(\ref{e3})$ for any pair of i.i.d. random variables $X, Y$.
\end{theorem}
\medskip

\begin{proof} Since 
\[
   F(\lambda)= a+b\lambda^2+\int_0^\infty \left(1-e^{-t\lambda^2} \right) \mu(dt),
\]
it is enough to check (\ref{e3}) for elementary functions $F_t(\lambda):=-e^{-t\lambda^2}$. Indeed, from the
identity
\[
  e^{-t|x-y|^2}- e^{-t|x+y|^2} = 2e^{-tx^2-ty^2} \sum_{n=0}^\infty \frac{(2txy)^{2n+1}}{(2n+1)!}
\]
we infer that
\[
   \E \left[ F_t(|X+Y|)- F_t(|X-Y|)\right] =  2
   \sum_{n=0}^\infty \frac{(2t)^{2n+1}}{(2n+1)!} \left[\E \left(X^{2n+1} e^{-tX^2}\right) \right]^2 
   \ge 0.
 \] 
\end{proof}
\medskip

Theorem 1 is a particular case of Theorem 2 because $G(\lambda)=\lambda^{\alpha}$ are Bernstein functions whenever 
$0<\alpha<1$. 
\medskip

{\bf Acknowledgement.}
The authors are grateful to H.Kempka, A. Koldobskii, W. Linde, and A. Nekvinda
for setting the problem and for valuable discussions.

The research was supported by the RFBR-DFG grant 09-01-91331, RFBR
grants 10-01-00154a, 10-01-00397a,  as well as by Federal Focused
Programme 2010-1.1-111-128-033.

\end{document}